\documentclass{article}

\usepackage{arxiv}

\usepackage[utf8]{inputenc} 
\usepackage[T1]{fontenc}    
\usepackage{hyperref}       
\usepackage{url}            
\usepackage{booktabs}       
\usepackage{amsfonts}       
\usepackage{nicefrac}       
\usepackage{microtype}      
\usepackage{graphicx}
\usepackage{amsmath}
\usepackage{amsthm}
\usepackage{color,soul}
\newtheorem{thm}{Theorem}[section]

\newtheorem{prop}[thm]{Proposition}
\newtheorem{cor}[thm]{Corollary}
\newtheorem{exmp}{Example}[section]

\title{Finite, fiber- and orientation-preserving actions on orientable Seifert manifolds with non-orientable base space}

\author{
  Benjamin Peet \\
  Department of Mathematics\\
  St. Martin's University\\
  Lacey, WA 98503 \\
  \texttt{bpeet@stmartin.edu} \\
}

\begin{document}
\maketitle

\begin{abstract}
This paper extends the results from the author's previous paper to consider finite, fiber- and orientation- preserving group actions on closed, orientable Seifert manifolds $M$ that fiber over a non-orientable base space. An orientable base space double cover $\tilde{M}$ of $M$ is constructed and then an isomorphism between the fiber- and orientation-preserving diffeomorphisms of ${M}$ and the fiber- and orientation-preserving actions on $\tilde{M}$ that preserve the orientation on the fibers and commute with the covering translation is shown. This result and previous results lead to a construction of all actions that satisfy a condition on the obstruction class and the structure of the finite groups that can act on $M$.
\end{abstract}



\section{Introduction and preliminary definitions}

\subsection{Introduction}

In the author's previous paper \textit{Finite, fiber- and orientation-preserving group actions on orientable Seifert manifolds with orientable base space} \cite{peet2018}, a construction of a finite, fiber- and orientation-preserving group action on a given Seifert manifold with an orientable base space was shown. This construction is founded upon the way a Seifert manifold is put together as Dehn fillings of $S^{1}\times F$. Here $F$ is a surface with boundary. The construction is - in a general sense - to take a product action on $S^{1}\times F$ and extend across the Dehn fillings. We refer to these actions as \textit{extended product actions}.

The main result in \cite{peet2018} shows that given a finite, orientation and fiber-preserving action on a Seifert manifold with an orientable base space, the action can be constructed as an extended product action - provided it satisfies a condition on the obstruction class of the Seifert manifold.

In this paper we consider extending these results to the non-orientable base space case. This is done by constructing an orientable base space double cover $\tilde{M}$ of $M$  and then showing an isomorphism between the fiber- and orientation-preserving diffeomorphisms of $\tilde{M}$ and the fiber- and orientation-preserving actions on $M$ that preserve the orientation on the fibers and commute with the covering translation.

Specifically, we prove:

\newtheorem*{thm:associativity}{Theorem \ref{thm:associativity}}
\begin{thm:associativity}
 Let $M$ be an orientable Seifert fibered manifold that fibers over an orbifold $B$ that has non-orientable underlying space and $q:\tilde{M}\rightarrow M$ be the orientable base space double cover. Let the covering translation be $\tau:\tilde{M}\rightarrow\tilde{M}$. Then there exists an isomorphism between $Diff_{+}^{fp}(M)$ and $Cent_{+}^{fop}(\tau)$.
\end{thm:associativity}

Finite groups of automorphisms on Seifert manifolds have been considered by J. Kalliongis and A. Miller in the particular case of Lens spaces in \textit{Geometric group actions on lens spaces} \cite{kalliongis2002geometric} and in \textit{Actions on lens spaces which respect a Heegaard decomposition} \cite{kalliongis2003actions}. J. Kalliongis and R. Ohashi considered another specific case of prism manifolds in \textit{Finite group actions on prism manifolds which preserve a Heegaard Klein bottle} \cite{kalliongis2011finite}. Lens spaces double cover prism manifolds and lens spaces have orientable base space whilst prism manifolds can be fibered over a nonorientable base space ($\mathcal{P}^{2}$ with a cone point). This motivates the work of this paper to consider how the actions on the lens spaces relate to the actions on the prism manifold and more generally, how the actions on a manifold with nonorientable base space relate to the actions on its double cover with orientable base space.

Theorem 3.1 allows us to see a construction of actions on a Seifert manifold with a non-orientable base space as a projected extended product action and shows that provided satisfaction of the obstruction condition on the lifted action on $\tilde{M}$ all such actions are constructed this way.

In particular, a fiber- and orientation-preserving action on an orientable Seifert manifold with non-orientable base space can be constructed as a projected action of the extended product action type shown in \cite{peet2018} provided the lifted action satisfies the obstruction condition.

The structure of the groups that can act this way is then considered and presented in Section 5. The specific result is as follows:

\newtheorem*{cor:associativity2}{Corollary \ref{cor:associativity2}}
\begin{cor:associativity2}
 Suppose that $\varphi:G\rightarrow Diff(M)$ is a finite group action on an orientable Seifert manifold with a non-orientable base space. Then provided that the unique lifted group action $\tilde{\varphi}:G\rightarrow Diff(\tilde{M})$ satisfies the obstruction condition, $G$ is isomorphic to a subgroup of $\mathbb{Z}_{2}\times H$ where $H$ is a finite group that acts orientation-preservingly on the orientable base space of $\tilde{M}$.
\end{cor:associativity2}

\subsection{Preliminary definitions}

We now give some preliminary definitions.
Let $M$ be an oriented smooth manifold of dimension $3$ without boundary and $G$ be a finite group. We let $Diff(M)$ be the group of self-diffeomorphisms of $M$, and then define a $G$-action on $M$ to be an injection $\varphi:G\rightarrow Diff(M)$. We use the notation $Diff_{+}(M)$ for the group of orientation-preserving self-diffeomorphisms of $M$. 

$M$ will be assumed to be an orientable Seifert-fibered manifold. We use the original Seifert definition. That is, a Seifert manifold is a $3$-manifold such that $M$ can be decomposed into disjoint fibers where each fiber is a simple closed curve. Then for each fiber $\gamma$, there exists a fibered neighborhood (that is, a subset consisting of fibers and containing $\gamma$) which can be mapped under a fiber-preserving map onto a solid fibered torus. Seifert manifolds were first introduced by H. Seifert in his dissertation \textit{Topologie Dreidimensionaler Gefaserter R{\"a}ume} \cite{Seifert1933}.

We call a Seifert bundle a Seifert manifold $M$ along with a continuous map $p:M\rightarrow B$ where $p$ identifies each fiber to a point. Here $B$ is an $2$-orbifold without boundary or mirror lines - that is some underlying space $B_{U}$ along with some cone points of various degrees referring to the critical fibers. Given an orientable (and oriented) base space it is possible to remove the critical fibers; take a section of the bundle; and then orient the fibers according to the normal vector to the surface.

A $G$-action is said to be fiber-preserving on a Seifert manifold $M$ if for any fiber $\gamma$ and any $g\in G$, $\varphi(g)(\gamma)$ is some fiber of $M$. We use the notation $Diff^{fp}(M)$ for the group of fiber-preserving self-diffeomorphisms of $M$ (given some Seifert fibration). We use the terminology fiber-orientation-preserving, if an action is fiber-preserving and preserves the orientation of the fibers. For this we use the notation $Diff^{fop}(M)$. Note that from above this is only in the case of an orientable (and oriented) base space. Given a fiber-preserving $G$-action, there is an induced action $\varphi_{B_{U}}:G_{B_{U}}\rightarrow Diff(B_{U})$ on the underlying space $B_{U}$ of the base space $B$. This is given by $\varphi_{B_{U}}(g)=p\circ\varphi(g)\circ p^{-1}$ for each $g\in G$ and then $\varphi_{B_{U}}(G)=G_{B_{U}}$. Note that this is well defined as for any given $b\in B$ and distinct $x_{1},x_{2}\in M$, by definition $(p\circ\varphi(g))(x_{1})=(p\circ\varphi(g))(x_{1})$.

Given a finite action $\varphi:G\rightarrow Diff^{fp}(M)$, we define the orbit number of a fiber $\gamma$ under the action to be $\#Orb_{\varphi}(\gamma)=\#\{\alpha|\varphi(g)(\gamma)=\alpha\textrm{ for some }g\in G\}$. 

If we have a manifold $M$, then a product structure on $M$ is a diffeomorphism $k:A\times B\rightarrow M$ for some manifolds $A$ and $B$. For further information we refer the reader to J.M. Lee's \textit{Introduction to Smooth Manifolds} \cite{lee2003smooth}. If a Seifert-fibered manifold $M$ has a product structure $k:S^{1}\times F\rightarrow M$ for some surface $F$ and $k(S^{1}\times\{x\})$ are the fibers of $M$ for each $x\in F$, then we say that $k:S^{1}\times F\rightarrow M$ is a fibering product structure of $M$. 

We say that a $G$-action $\varphi:G\rightarrow Diff(A\times B)$ is a product action if for each $g\in G$, the diffeomorphism $\varphi(g):A\times B\rightarrow A\times B$ can be expressed as $(\varphi_{1}(g),\varphi_{2}(g))$ where $\varphi_{1}(g):A\rightarrow A$ and $\varphi_{2}(g):B\rightarrow B$. Here $\varphi_{1}:G\rightarrow Diff(A)$ and $\varphi_{2}:G\rightarrow Diff(B)$ are not necessarily injections. 

Given an action $\varphi:G\rightarrow Diff(M)$ and a product structure $k:A\times B\rightarrow M$, we say that $\varphi$ leaves the product structure $k:A\times B\rightarrow M$ invariant if $\psi(g)=k^{-1}\circ\varphi(g)\circ k$ defines a product action $\psi:G\rightarrow Diff(A\times B)$.

\section{Seifert manifolds with non-orientable base space}

Orientable Seifert manifolds that fiber over a orientable base space $B$ can be constructed by taking an $S^1$ bundle over an orientable surface with boundary and Dehn filling the torus boundary components. This yields a Seifert manifold that we denote by $$(g,o_{1}|(q_{1},p_{1}),\ldots,(q_{n},p_{n})), q_{i}>0$$.
Here $o_{1}$ denotes the orientability of $F$ and $g$ denotes its' genus. $(q_{i},p_{i})$ denote the fillings.

Orientable Seifert manifolds can also however fiber over a non-orientable base space $B$. An example is of the prism manifolds which can fiber over $\mathbb{P}^{2}(n)$, a projective plane with a cone point. 

We use the notation $$(g,n_{2}|(q_{1},p_{1}),\ldots,(q_{n},p_{n})), q_{i}>0$$
where $n_{2}$ indicates that the underlying surface of the base space is diffeomorphic to a connected sum of $g$ copies of $\mathbb{P}^{2}$. 
The notation here is an accepted adaptation of the original Seifert notation, and note that $n_{2}$ (and $o_{1}$) are simply symbols indicating the orientability of the Seifert manifold and the orientability of the base space.

Following M. Jankins and W. D. Neumann's \textit{Lectures on Seifert manifolds} \cite{jankins1983lectures}, we present a construction of such a manifold.

Let $S$ be a non-orientable surface of genus $g$, then $S=S_{1}\#S_{2}$ where $S_{1}$ is orientable and $S_{2}$ is either $\mathbb{P}^{2}$ or $\mathbb{P}^{2}\#\mathbb{P}^{2}$. This follows from the Classification Theorem for surfaces, see L.C. Kinsey's \textit{Topology of surfaces} \cite{kinsey1997topology}.

So then we can cut $S$ along one or two simple closed curves so that it is decomposed into $S_{1}$ and $S_{2}$ or decomposed into $S_{1}$, $S_{2}$ and $S_{3}$. Here $S_{1}$ is an orientable surface with one or two boundary components and $S_{2}$ and $S_{3}$ are Möbius bands. In the first case see Figure 1, in the second see Figure 2.

\begin{figure}[ht]
\centering
\includegraphics[height=2cm]{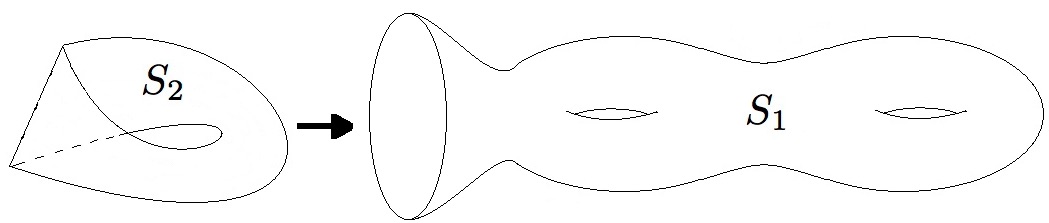}
\caption{A non-orientable surface decomposed into $S_{1}$, an orientable surface with one boundary component, and $S_{2}$, a Möbius band.}
\end{figure}

\begin{figure}[ht]
\centering
\includegraphics[height=2cm]{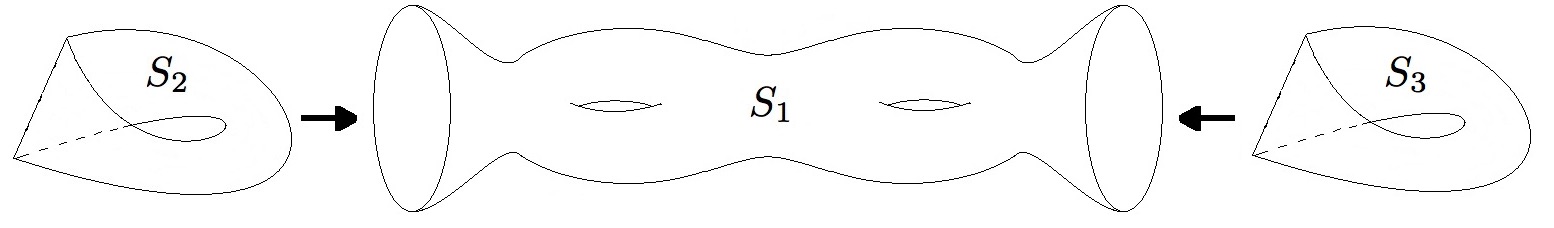}
\caption{A non-orientable surface decomposed into $S_{1}$, an orientable surface with two boundary components, and $S_{2}$ and $S_{3}$, both Möbius bands.}
\end{figure}

By \cite{jankins1983lectures}, there is a unique fibered and orientable $3$-manifold that fibers over the Möbius band - it is the twisted $I$-bundle over the Klein bottle fibered meridianally. This is double covered by $S^{1}\times A$, where $A$ is an annulus. We let $M_{2}$ fiber over $S_{2}$ and $M_{3}$ fiber over $S_{3}$. $M_{1}$ is constructed by taking a trivially fibered $S^{1}\times S_{1}$ and drilling $n$ fibers and refilling according to the invariants $(q_{i},p_{i})$. 

Now, $M_{1}$ has a torus boundary and can be framed longitunally by a fiber and meridianally by a section to the bundle. $M_{2}$ has either one or two tori in its' boundary and each can framed longitudinally by a fiber and meridianally by a section to the bundle (necessarily before the drilling and filling).

The manifold $M$ represented by $$(g,n_{2}|(q_{1},p_{1}),\ldots,(q_{n},p_{n})), q_{i}>0$$ is then obtained by gluing the torus boundaries of $M_{1}$ and $M_{2}$, or $M_{1}$ and $M_{2}$ and $M_{3}$ by attaching longitudes to longitudes and meridians to meridians. That is, according to the identity matrix between the first homology groups generated by representatives of the longitudes and meridians.

We now consider a particular double cover of a Seifert manifold that fibers over a non-orientable base space. This well known result is noted in W. D. Neumann and F. Raymond's paper \textit{Seifert manifolds, plumbing, $\mu$-invariant and orientation reversing maps} \cite{neumann1978seifert}, but we show it here constructively as the construction is utilized later on in the proof of our main result Theorem 3.1.

\begin{prop}
${M=(g+1,n_{2}|(q_{1},p_{1}),\ldots,(q_{n},p_{n}))}$ is double covered by $\hfill$ $\tilde{M}=(g,o_{1}|(q_{1},p_{1}),(q_{1},p_{1}),\ldots,(q_{n},p_{n}),(q_{n},p_{n}))$.
\end{prop}

\begin{proof}
 
The underlying space of $\tilde{B}$ is a connected sum of $g$ tori. We consider two cases, when $g$ is even or odd.

\underline{Case 1:} g is even

We cut the underlying space of $\tilde{B}$ along two simple closed curves to leave $\tilde{B}_{1}\cong A$ and $\tilde{B}_{2},\tilde{B}_{3}$ which are connected sums of $\frac{g}{2}$ tori with a disc removed. See Figure 3 for the case where $g=2$.

\begin{figure}[ht]
\centering
\includegraphics[height=2cm]{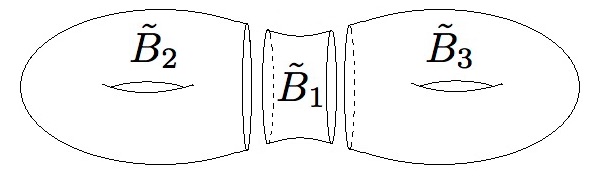}
\caption{$\tilde{B}$ (in the case where $g=2$) decomposed into $\tilde{B}_{1},\tilde{B}_{2},\tilde{B}_{3}$. (Cone points not shown).}
\end{figure}

Now we can assume that the cone points of $\tilde{B}$ lie evenly between $\tilde{B}_{2}$ and $\tilde{B}_{3}$. So let $\tilde{M}_{1}$ and $\tilde{M}_{2},\tilde{M}_{3}$ be the cut $3$-manifolds that fiber over $\tilde{B}_{1}$ and $\tilde{B}_{2},\tilde{B}_{3}$. Then $\tilde{M}_{1}\cong S^{1}\times A$ and $\tilde{M}_{2},\tilde{M}_{3}$ are constructed by taking a genus $\frac{g}{2}$ surface with a disc removed crossing with $S^{1}$ and replacing $n$ fibers in each according to the given Seifert invariants. 

So now $\tilde{M}_{1}$ double covers the twisted $I$-bundle over the Klein bottle which we denote as $M_{1}$. $\tilde{M}_{2},\tilde{M}_{3}$ double cover the manifold $M_{2}$ obtained by taking a genus $\frac{g}{2}$ surface with one disc removed cross $S^{1}$ and replacing $n$ fibers according to half of the given Seifert invariants. This in fact a $2$-sheeted cover. 

Regluing along the torus boundary gives the required double cover. The underlying surface of $B$ is therefore a connected sum of $\mathbb{P}^{2}$ and a genus $\frac{g}{2}$ surface. This is a connected sum of $g+1$ copies of $\mathbb{P}^{2}$.

\underline{Case 2:} g is odd

This time, we cut the underlying space of $\tilde{B}$ along four simple closed curves to leave $\tilde{B}_{1},\tilde{B}_{2}\cong A$ and $\tilde{B}_{3},\tilde{B}_{4}$ connected sums of $\frac{g-1}{2}$ tori with two discs removed. See Figure 4 for the case where $g=3$.

\begin{figure}[ht]
\centering
\includegraphics[height=3cm]{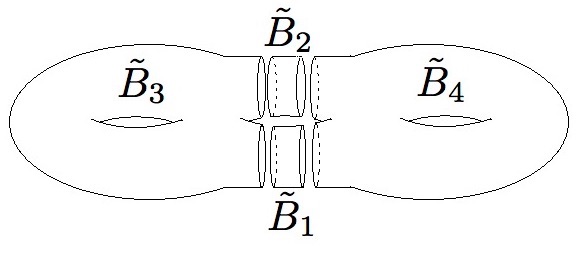}
\caption{$\tilde{B}$ (in the case where $g=3$) decomposed into $\tilde{B}_{1},\tilde{B}_{2},\tilde{B}_{3},\tilde{B}_{4}$. (Cone points not shown)}
\end{figure}

Now we can again assume that the cone points of $\tilde{B}$ lie evenly between $\tilde{B}_{3},\tilde{B}_{4}$. Label $\tilde{M}_{1}$ and $\tilde{M}_{2}$ as the cut $3$-manifolds that fiber over $\tilde{B}_{1}$ and $\tilde{B}_{2}$. Then $\tilde{M}_{1},\tilde{M}_{2}\cong S^{1}\times A$. Label $\tilde{M}_{3},\tilde{M}_{4}$ as the cut $3$-manifolds that fiber over $\tilde{B}_{3}$ and $\tilde{B}_{4}$. Then $\tilde{M}_{3},\tilde{M}_{4}$ are each constructed from taking a genus $\frac{g-1}{2}$ surface with two discs removed cross $S^{1}$ and replacing $n$ fibers according to the given Seifert invariants. 

$\tilde{M}_{1}$ and $\tilde{M}_{2}$ double cover two disjoint twisted $I$-bundles over the Klein bottle which we denote as $M_{1}$ and $M_{2}$. $\tilde{M}_{3}$ and $\tilde{M}_{4}$ double cover the manifold $M_{3}$ obtained by taking a genus $\frac{g-1}{2}$ surface with two discs removed cross $S^{1}$ and replacing $n$ fibers according to half of the given Seifert invariants. Again, this is in fact a $2$-sheeted cover. 

Regluing along the two torus boundaries gives the required double cover. The underlying surface of $B$ is therefore a connected sum of $\mathbb{P}^{2}\#\mathbb{P}^{2}$ and a genus $\frac{g-1}{2}$ surface. This is a connected sum of $g+1$ copies of $\mathbb{P}^{2}$.

\end{proof}

\begin{exmp}

This proposition is a general case of the known result of J. Kalliongis and R. Ohashi in \textit{Finite group actions on prism manifolds which preserve a Heegaard Klein bottle} \cite{kalliongis2011finite}, that a prism manifold $M=(1,n_{2}|(q,p))$ is double covered by a lens space $\tilde{M}=(0,o_{1}|(q,p),(q,p))$.

\end{exmp}

We henceforth use the term \textit{orientable base space double cover} to indicate the double cover shown in the previous proposition. 

We now quote Theorem 1.1. from \cite{neumann1978seifert} regarding Seifert invariants:

\begin{thm}

Let $M$ and $M'$ be two orientable Seifert manifolds with associated Seifert invariants $(g,o_{1}|(q_{1},p_{1}),\ldots,(q_{s},p_{s}))$ and $(g,o_{1}|(q_{1}',p_{1}'),\ldots,(q_{t}',p_{t}'))$ respectively. Then $M$ and $M'$ are orientation-preservingly diffeomorphic by a fiber-preserving diffeomorphism if and only if, after reindexing the Seifert pairs if necessary, there exists an n such that:

i) $q_{i}=q_{i}'$ for $i=1,\ldots,n$ and $q_{i}=q_{j}'=1$ for $i,j>n$ 

ii) $p_{i}\equiv p_{i}'\textrm{ (mod }q_{i})$ for $i=1,\ldots,n$

iii) $$\sum_{i=1}^{s}\frac{p_{i}}{q_{i}}=\sum_{i=1}^{t}\frac{p_{i}'}{q_{i}'}$$

\end{thm}

From this and Proposition 2.1, we get:

\begin{cor}

Let $M$ and $M'$ be two orientable Seifert manifolds with associated Seifert invariants $(g,n_{2}|(q_{1},p_{1}),\ldots,(q_{s},p_{s}))$ and $(g,n_{2}|(q_{1}',p_{1}'),\ldots,(q_{t}',p_{t}'))$ respectively. Then $M$ and $M'$ are orientation-preservingly diffeomorphic by a fiber-preserving diffeomorphism if and only if, after reindexing the Seifert pairs if necessary, there exists an $n$ such that:

i) $q_{i}=q_{i}'$ for $i=1,\ldots,n$ and $q_{i}=q_{j}'=1$ for $i,j>n$ 

i) $p_{i}\equiv p_{i}'\textrm{ (mod }q_{i})$ for $i=1,\ldots,n$

iii) $$\sum_{i=1}^{s}\frac{p_{i}}{q_{i}}=\sum_{i=1}^{t}\frac{p_{i}'}{q_{i}'}$$

\end{cor}

It is therefore possible to define the normalized form as $(g,n_{2}|(q_{1},p_{1}),\ldots,(q_{n},p_{n}),(1,b))$ where $0<p_{i}<q_{i}$ and $b$ is an integer called the \textit{obstruction class}. 

\section{Isomorphism between $Diff_{+}^{fp}(M)$ and $Cent_{+}^{fop}(\tau)$}

The main goal of this section is to show a correspondence between the finite, fiber- and orientation-preserving actions on $M$ and a subset of the finite, fiber-preserving actions on $\tilde{M}$, the orientable base space double cover. 

We state and prove the following proposition which establishes a one-to-one correspondence between the finite, fiber-preserving actions on $M$ and the finite, fiber-orientation-preserving actions on $\tilde{M}$ that commute with the centralizer of the covering translation, that is $Cent_{+}^{fop}(\tau)=\{f\in Diff_{+}^{fop}(\tilde{M})|f\circ \tau = \tau \circ f\}$.

\begin{thm}
\label{thm:associativity}
Let $M$ be an orientable Seifert fibered manifold that fibers over an orbifold $B$ that has non-orientable underlying space and $q:\tilde{M}\rightarrow M$ be the orientable base space double cover. Let the covering translation be $\tau:\tilde{M}\rightarrow\tilde{M}$. Then there exists an isomorphism between $Diff_{+}^{fp}(M)$ and $Cent_{+}^{fop}(\tau)$.
\end{thm}

\begin{proof}
We begin with the following diagram:

$$\begin{array}{ccccc}
 &  & q\\
 & M & \leftarrow & \tilde{M}\\
p & \downarrow &  & \downarrow & \tilde{p}\\
 & B & \leftarrow & \tilde{B}\\
 &  & q_{B}
\end{array} $$

Here, $p,\tilde{p}$ are the bundle maps for $M,\tilde{M}$  respectively, and $q,q_{B}$ are the covering projections for $M,B$ respectively.

The first thing to note is that the finite group actions on $B$ are in one-to-one correspondence with the orientation-preserving group actions on $\tilde{B}$. This is due to T.W. Tucker in his paper \textit{Finite groups acting on surfaces and the genus of a group} \cite{tucker1983finite}.

Now we pick $f\in Diff_{+}^{fp}(M)$ and are required to find $\tilde{f}:\tilde{M}\rightarrow\tilde{M}$, which will be our desired lift of $f$.

Taking $\tilde{x}_{0} \in \tilde{M}$, $x_{0}=(f\circ q)(\tilde{x}_{0})\in M$, and some lift $\tilde{x}_{1} \in \tilde{M}$ of $x_{0}$ as basepoints of $M$ and $\tilde{M}$ we have the following diagram:

$$\begin{array}{ccccc}
 &  & \tilde{f} \\
 & (\tilde{M},\tilde{x}_{0}) & \dashrightarrow & (\tilde{M},\tilde{x}_{1})\\
q & \downarrow &  & \downarrow & q\\
 & (M,q(\tilde{x}_{0})) & \rightarrow & (M,x_{0})\\
 &  & f
\end{array} $$

This induces the following diagram on the fundamental groups:

$$\begin{array}{ccccc}
 &  & \tilde{f}_{*} \\
 & \pi_{1}(\tilde{M},\tilde{x}_{0}) & \dashrightarrow & \pi_{1}(\tilde{M},\tilde{x}_{1})\\
q^{0}_{*} & \downarrow &  & \downarrow & q^{1}_{*}\\
 & \pi_{1}(M,q(\tilde{x}_{0})) & \rightarrow & \pi_{1}(M,x_{0})\\
 &  & f_{*}
\end{array} $$

We proceed to show that $f_{*}(q^{0}_{*}(\pi_{1}(\tilde{M},\tilde{x}_{0}))) \subset q^{1}_{*}(\pi_{1}(\tilde{M},\tilde{x}_{1}))$ which is the sufficient condition to lift $f$ according to the lifting criterion. For further reference, see A. Hatcher's book \textit{Algebraic Topology} \cite{hatcher2002algebraic}.

We have that $M=(g+1,n_{2}|(q_{1},p_{1}),\ldots,(q_{n},p_{n}))$ and break into the two cases where $g$ is even and odd.

We use the notation and representation of the following fundamental groups from J. Hempel's book \textit{$3$-Manifolds} \cite{hempel20043}, as well as M. Boileau, S. Maillot, and J. Porti's paper \textit{Three-dimensional orbifolds and their geometric structures} \cite{boileau2003three}.

\underline{Case 1:} g is even
\begin{align*}
\pi_{1}(M,x_{0})=	&\big\langle a_{1},b_{1}\ldots, a_{\frac{g}{2}},b_{\frac{g}{2}},x,c_{1},\ldots,c_{n},t|a_{i}t=ta_{i},b_{i}t=tb_{i}, c_{i}t=tc_{i}, \\ & \quad xt=t^{-1}x,c_{j}^{q_{j}}t^{p_{j}}=1,c_{1}\cdots c_{n}[a_{1},b_{1}]\cdots[a_{\frac{g}{2}},b_{\frac{g}{2}}]x^{-2}=1 \big\rangle \\
\pi_{1}^{orb}(B,p(x_{0}))=&	\big\langle a'_{1},b'_{1}\ldots,a'_{\frac{g}{2}},b'_{\frac{g}{2}},x',c'_{1},\ldots,c'_{n}|c_{j}^{\prime q_{j}}=1,c'_{1}\cdots c'_{n}[a'_{1},b'_{1}]\cdots[a'_{\frac{g}{2}},b'_{\frac{g}{2}}]x^{\prime-2}=1\big\rangle
\end{align*} Here we have that $p_{*}(a_{i})=a'_{i}, p_{*}(b_{i})=b'_{i}, p_{*}(c_{i})=c'_{i}$, and $p_{*}(x)=x'$.

The only generator of $\pi_{1}^{orb}(B,p(x_{0}))$ that represents an orientation-reversing loop is $x'$. This follows from the fact that $\pi_{1}^{orb}(B,p(x_{0}))$ is the fundamental group of the underlying space with the extra generators $c_{1},\ldots,c_{n}$ and associated relations. Then reference to a standard text on Algebraic Topology such as J. Munkres' \textit{Topology} \cite{munkres} shows that the fundamental group of a nonorientable surface without boundary can be generated $a'_{1},b'_{1}\ldots,a'_{\frac{g}{2}},b'_{\frac{g}{2}}$ each of which can be represented by orientation-preserving loops and $x'$ represented by an orientation-reversing loop. Note that $c_{1},\ldots,c_{n}$ are also represented by orientation-preserving loops.

\underline{Case 2:} g is odd
\begin{align*}
\pi_{1}(M,x_{0})=	& \big\langle a_{1},b_{1}\ldots,a_{\frac{g-1}{2}},b_{\frac{g-1}{2}},x,y,c_{1},\ldots,c_{n},t|a_{i}t=ta_{i},b_{i}t=tb_{i}, c_{i}t=tc_{i}, \\ & \quad xt=t^{-1}x,yt=ty,c_{j}^{q_{j}}t^{p_{j}}=1,c_{1}\cdots c_{n}[a_{1},b_{1}]\cdots[a_{\frac{g-1}{2}},b_{\frac{g-1}{2}}]xyx^{-1}y=1\big\rangle \\ 
\pi_{1}^{orb}(B,p(x_{0}))=&	\big\langle a'_{1},b'_{1}\ldots,a'_{\frac{g-1}{2}},b'_{\frac{g-1}{2}},x',y',c'_{1},\ldots,c'_{n}|c_{j}^{\prime q_{j}}=1,\\ &
\quad c'_{1}\cdots c'_{n}[a'_{1},b'_{1}]\cdots[a'_{\frac{g-1}{2}},b'_{\frac{g-1}{2}}]x'y'x^{\prime-1}y'=1\big\rangle
\end{align*} With $p_{*}(a_{i})=a'_{i}$, $p_{*}(b_{i})=b'_{i}$, $p_{*}(c_{i})=c'_{i}$, $p_{*}(x)=x'$ and $p_{*}(y)=y'$. Again, the only generator of $\pi_{1}^{orb}(B,p(x_{0}))$ that represents an orientation-reversing loop is $x'$. 

We now consider the orientable base space double cover $\tilde{M}=(g,o_{1}|(q_{1},p_{1}),(q_{1},p_{1}),\ldots,(q_{n},p_{n}),(q_{n},p_{n}))$ and again using \cite{hempel20043} we have:\begin{align*}
\pi_{1}(\tilde{M},\tilde{x}_{1})=	& \big\langle \tilde{a}_{1},\tilde{b}_{1},\ldots,\tilde{a}_{g},\tilde{b}_{g},\tilde{c}_{1},\ldots,\tilde{c}_{2n},\tilde{t}|\tilde{a}_{i}\tilde{t}=\tilde{t}\tilde{a}_{i},\tilde{b}_{i}\tilde{t}=\tilde{t}\tilde{b}_{i},\\ & \quad \tilde{c}_{i}\tilde{t}=\tilde{t}\tilde{c}_{i},\tilde{c}_{j}^{q_{j}}\tilde{t}^{p_{j}}=1,\tilde{c}_{j+n}^{q_{j}}\tilde{t}^{p_{j}}=1,\tilde{c}_{1}\cdots\tilde{c}_{2n}[\tilde{a}_{1},\tilde{b}_{1}]\cdots[\tilde{a}_{g},\tilde{b}_{g}]=1 
\big\rangle
\end{align*}Note that the covering translation $\tau:\tilde{M}\rightarrow\tilde{M}$ leaves invariant either a torus that separates $\tilde{M}$ into two diffeomorphic halves or a pair of tori that together separate $\tilde{M}$ into two diffeomorphic halves. This is respectively in the cases where the genus of $\tilde{B}$ is even or odd and follows from Proposition 2.1.

If we call these two halves $\tilde{M}_{1}$ and $\tilde{M}_{2}$, then $\tau:\tilde{M}\rightarrow\tilde{M}$ exchanges $\tilde{M}_{1}$ and $\tilde{M}_{2}$. The restricted projection can then be taken so that:

\begin{enumerate}
\item
$q^{1}_{*}(\tilde{t})=t$ and $q^{1}_{*}(\tilde{c}_{i})=c_{i}$ for $i=1,\ldots,n$.

\item
For $g$ even, $q^{1}_{*}(\tilde{a}_{i})=a_{i}$ and $q^{1}_{*}(\tilde{b}_{i})=b_{i}$ for $i=1,\ldots,\frac{g}{2}$.

\item
For $g$ odd, $q^{1}_{*}(\tilde{a}_{i})=a_{i}$, $q^{1}_{*}(\tilde{b}_{i})=b_{i}$ for $i=1,\ldots,\frac{g-1}{2}$ and $q^{1}_{*}(\tilde{a}_{\frac{g+1}{2}})=x^{2}$, $q^{1}_{*}(\tilde{b}_{\frac{g+1}{2}})=y$. 

\end{enumerate}

This again follows from Proposition 2.1 where we assume that the base point $\tilde{x}_{1}$ lies in $\tilde{M}_{1}$ and  - for $g$ even - $\tilde{a}_{1},\tilde{b}_{1},\ldots,\tilde{a}_{\frac{g}{2}},\tilde{b}_{\frac{g}{2}},\tilde{c}_{1},\ldots,\tilde{c}_{n},\tilde{t}$ are represented by loops that lie in $\tilde{M}_{1}$. For $g$ odd, $\tilde{a}_{1},\tilde{b}_{1},\ldots,\tilde{a}_{\frac{g-1}{2}},\tilde{b}_{\frac{g-1}{2}},\tilde{c}_{1},\ldots,\tilde{c}_{n},\tilde{t}$ are represented by loops that lie in $\tilde{M}_{1}$.

Now suppose that $\tilde{d}$ is one of $\tilde{a}_{1},\tilde{b}_{1},\ldots,\tilde{a}_{g},\tilde{b}_{g},\tilde{c}_{1},\ldots,\tilde{c}_{2n}$ and note that $p_{*}(q^{1}_{*}(\tilde{d}))=p_{B*}(\tilde{p}_{*}(\tilde{d}))$ cannot be orientation-reversing as the only generator of $\pi_{1}^{orb}(B,p(x_{0}))$ that represents an orientation-reversing loop is $x'$. Hence $q^{1}_{*}(\tilde{d})$ cannot be represented with a word involving a single power of $x$ and therefore $q^{1}_{*}(\pi_{1}(\tilde{M},\tilde{x}_{1}))$ has generators:
$$a_{1},b_{1}\ldots,a_{\frac{g}{2}},b_{\frac{g}{2}},c_{1}\ldots,c_{n}$$ 
in the even case and:
$$a_{1},b_{1}\ldots,a_{\frac{g}{2}},b_{\frac{g}{2}},c_{1}\ldots,c_{n},y$$
in the odd.

We now consider $f_{*}(q^{0}_{*}(\pi_{1}(\tilde{M},\tilde{x}_{0})))$ and use the following representation: \begin{align*}
\pi_{1}(\tilde{M},\tilde{x}_{0})=	& \big\langle \tilde{\alpha}_{1},\tilde{\beta}_{1},\ldots,\tilde{\alpha}_{g},\tilde{\beta}_{g},\tilde{\gamma}_{1},\ldots,\tilde{\gamma}_{2n},\tilde{\delta}|\tilde{\alpha}_{i}\tilde{\delta}=\tilde{\delta}\tilde{\alpha}_{i},\tilde{\beta}_{i}\tilde{\delta}=\tilde{\delta}\tilde{\beta}_{i},\\ & \quad \tilde{\gamma}_{i}\tilde{\delta}=\tilde{\delta}\tilde{\gamma}_{i},\tilde{\gamma}_{j}^{q_{j}}\tilde{\delta}^{p_{j}}=1,\tilde{\gamma}_{j+n}^{q_{j}}\tilde{\delta}^{p_{j}}=1,\tilde{\gamma}_{1}\cdots\tilde{\gamma}_{2n}[\tilde{\alpha}_{1},\tilde{\beta}_{1}]\cdots[\tilde{\alpha}_{g},\tilde{\beta}_{g}]=1 
\big\rangle
\end{align*} For $g$ even, as $f$ is fiber-preserving and orientation-preserving, we have that $f_{*}(q^{1}_{*}(\tilde{\gamma}))=t^{\pm1}\in q^{0}_{*}(\pi_{1}(\tilde{M},x_{0}))$. Also, $f \circ q$ will send each of $\alpha_{1},\beta_{1}\ldots,\alpha_{\frac{g}{2}},\beta_{\frac{g}{2}},\gamma_{1}\ldots,\gamma_{n}$  to some word $w$ on $a_{1},b_{1}\ldots,a_{\frac{g}{2}},b_{\frac{g}{2}},c_{1}\ldots,c_{n}$ multiplied by $t$ to some power. This follows from the fact that each of $\alpha_{1},\beta_{1}\ldots,\alpha_{\frac{g}{2}},\beta_{\frac{g}{2}},\gamma_{1}\ldots,\gamma_{n}$ project to orientation-preserving loops under $q$ and $f$ is orientation-preserving. Then $wt^{d}\in q^{0}_{*}(\pi_{1}(\tilde{M},x_{0}))$.

For $g$ odd, $f$ is fiber-preserving and orientation-preserving, we have that $f_{*}(q^{1}_{*}(\tilde{\gamma}))=t^{\pm1}\in q^{0}_{*}(\pi_{1}(\tilde{M},x_{0}))$. Also, $f \circ q$ will send each of $\alpha_{1},\beta_{1}\ldots,\alpha_{\frac{g}{2}},\beta_{\frac{g}{2}},\gamma_{1}\ldots,\gamma_{n}$ will be sent to some word $w$ on $a_{1},b_{1}\ldots,a_{\frac{g}{2}},b_{\frac{g}{2}},c_{1}\ldots,c_{n},y$ multiplied by $t$ to some power. This again follows from the fact that each of $\alpha_{1},\beta_{1}\ldots,\alpha_{\frac{g}{2}},\beta_{\frac{g}{2}},\gamma_{1}\ldots,\gamma_{n}$ project to orientation-preserving loops under $q$ and $f$ is orientation-preserving. Then $wt^{d}\in q^{0}_{*}(\pi_{1}(\tilde{M},x_{0}))$.

So that $f_{*}(q^{0}_{*}(\pi_{1}(\tilde{M},\tilde{x}_{0}))) \subset q^{1}_{*}(\pi_{1}(\tilde{M},\tilde{x}_{1}))$ holds.

We have proved the existence of a lift, but now $f$ can lift to some $\tilde{f}$ and $\tilde{f}\circ\tau$. Only one of these is orientation-preserving on the base space. We take $\tilde{f}$ to be the lift that is orientation-preserving on the base space.

This defines the correspondence from $Diff_{+}^{fp}(M)$ to $Cent_{+}^{fop}(\tau)$.

We now show that this is a homomorphism. Take $f_{1},f_{2}\in Diff_{+}^{fp}(M)$ and calculate: \begin{align*}
q\circ\widetilde{f_{1}\circ f_{2}}	&=f_{1}\circ f_{2}\circ q\\
	&=f_{1}\circ q\circ\widetilde{f_{2}}\\
	&=q\circ\tilde{f_{1}}\circ\tilde{f}_{2}
\end{align*} It then follows that either $\widetilde{f_{1}\circ f_{2}}=\tilde{f_{1}}\circ\tilde{f}_{2}$ or $\widetilde{f_{1}\circ f_{2}}=\tilde{f_{1}}\circ\tilde{f}_{2}\circ\tau$. The second case is not possible as we chose the lift to be orientation-preserving on the base space.

Hence it is a homomorphism.

To see injectivity, note that $(q\circ id)(x)=(f\circ q)(x)$ implies that $q(x)=f(q(x))$ and hence $f=id$.

For surjectivity, we note that as $\tilde{f}\in Cent_{+}^{fop}(\tau)$ we can project to some $f\in Diff_{+}^{fp}(M)$. This follows from standard covering space theory, again see \cite{munkres}.

Hence there is an isomorphism from $Cent_{+}^{fop}(\tau)$ to $Diff_{+}^{fp}(M)$. 

\end{proof}

From this we directly yield the corollary:

\begin{cor}
Let $M$ be an orientable Seifert fibered manifold that fibers over an orbifold $B$ that has non-orientable underlying space and $q:\tilde{M}\rightarrow M$ be the orientable base space double cover. Then there is a one-to-one correspondence between the finite, orientation and fiber-preserving group actions on $M$ and the finite orientation and fiber orientation-preserving group actions on $\tilde{M}$ that commute with the covering translation $\tau:\tilde{M}\rightarrow\tilde{M}$.
\end{cor}

\section{Using previous results}

We now give a summary of the construction of a finite, fiber- and orientation-preserving group action on a Seifert manifold with orientable base space $M=(g,o_{1}|(q_{1},p_{1}),\ldots,(q_{n},p_{n}))$ given in \cite{peet2018}:

We decompose $M$ into $\hat{M}$ and $X$ where $\hat{M}\cong S^{1}\times F$ is trivially fibered and $X$ is a disjoint union of $n$ solid tori. We then have a gluing map $d:\partial X\rightarrow\partial\hat{M}$, so that for a fibering product structure $k_{\hat{M}}:S^{1}\times F\rightarrow\hat{M}$, there is some $k_{X}:S^{1}\times(D_{1}\cup\ldots\cup D_{n})\rightarrow X$ and restricted positively oriented product structures $k_{\partial V_{i}}:S^{1}\times S^{1}\rightarrow\partial V_{i}$ and $k_{T_{i}}:S^{1}\times S^{1}\rightarrow T_{i}$ such that $(k_{T_{i}}^{-1}\circ d|_{\partial V_{i}}\circ k_{\partial V_{i}})(u,v)=(u^{x_{i}}v^{p_{i}},u^{y_{i}}v^{q_{i}})$.

We pick a finite, fiber-preserving group action on $\hat{M}$ by first choosing some (not-necessarily effective) group action $\varphi_{1}:G\rightarrow Diff(S^{1})$. This will necessarily be of the form: $$\varphi_{1}(g)(u)=\theta_{1}(g)u^{\alpha(g)}$$

Here $\theta_{1}:G\rightarrow S^{1}$ and $\alpha:G\rightarrow\{-1,1\}$. 

We then choose a (not-necessarily effective) group action $\varphi_{2}:G\rightarrow Diff(F)$ such that if we parameterize each component of $\partial F$ in the same way as in Section 2 and then express $\partial F=\{(v,i)|v\in S^{1},i\in\{1,\ldots,n\}\}$, we can write: $$\varphi_{2}(g)|_{\partial F}(v,i)=(\theta_{2}(i,g)v^{\alpha(g)},\beta(g)(i))$$
Here $\theta_{2}:\{1,\ldots,n\}\times G\rightarrow S^{1}$, and $\beta:G\rightarrow perm(\{1,\ldots,n\})$ are such that $\beta(g)(i)=j$ only if $(q_{i},p_{i})=(q_{j},p_{j})$.

The precise nature of each of these maps is shown in \cite{peet2018}.

Then we define our group action $\varphi:G\rightarrow Diff(\hat{M})$ by: $$(k_{\hat{M}}^{-1}\circ\varphi(g)\circ k_{\hat{M}})(u,x)=(\varphi_{1}(g)(u),\varphi_{2}(g)(x))$$
So now we can fully express $\varphi:G\rightarrow Diff(\hat{M})$ on the boundary of $\hat{M}$ by: $$(k_{T_{\beta(g)(i)}}^{-1}\circ\varphi(g)\circ k_{T_{i}})(u,v)=(\theta_{1}(g)u^{\alpha(g)},\theta_{2}(i,g)v^{\alpha(g)})$$
So we can now induce an action on $\partial X$ by: $$\psi:G\rightarrow Diff(\partial X),\psi(g)=d^{-1}\circ\varphi(g)|_{\partial\hat{M}}\circ d$$
We here note that:$$k_{X}^{-1}(X)=\{(u,v,i)|u\in S^{1},v\in D,i\in\{1,\ldots,n\}\}$$
Where $D$ is the unit disc. Hence the action $\psi:G\rightarrow Diff(X)$ straightforwardly extends by coning inwards.

So now we have defined finite, fiber-preserving actions on $\hat{M}$ and $X$ such that they agree under the gluing map $d:\partial X\rightarrow\partial\hat{M}$. This completes the construction.

We define actions that can be constructed in this way as \textit{extended product actions}. 

We can now state the main result from \cite{peet2018}:

\begin{thm}
 Let $M$ be an orientable Seifert $3$-manifold that fibers over an orientable base space. Let $\varphi:G\rightarrow Diff_{+}^{fp}(M)$ be a finite group action on $M$ such that the obstruction class can expressed as $$b=\sum_{i=1}^{m}(b_{i}\cdot\#Orb_{\varphi}(\alpha_{i}))$$ for a collection of fibers $\{\alpha_{1},\ldots,\alpha_{m}\}$ and integers $\{b_{1},\ldots,b_{m}\}$. Then $\varphi$ is an extended product action. 
 \end{thm}
 
 So now given an action $\tilde{\varphi}:G\rightarrow Diff_{+}^{fop}(\tilde{M})$ that satsifies the obstruction condition in Theorem 4.1, we can yield a decomposition of $\tilde{M}$ into $\widehat{\tilde{M}}$ and a collection of solid tori $\{V_{1},\ldots,V_{m}\}$ with a product structure $k:S^{1}\times F\rightarrow\widehat{\tilde{M}}$. Then there is a restricted action $\widehat{\tilde{\varphi}}:G\rightarrow Diff(\hat{\tilde{M}})$ such that each $(k^{-1}\circ\widehat{\tilde{\varphi}}(g)\circ k)(u,x)=(\widehat{\tilde{\varphi}}_{1}(g)(u),\widehat{\tilde{\varphi}}_{2}(g)(x))$ is a product map.
 
 This then leads us to the following result:

\begin{thm}
Let $M$ be an orientable Seifert fibered manifold that fibers over an orbifold $B$ that has non-orientable underlying space and obstruction class $b$. Let $q:\tilde{M}\rightarrow M$ be the orientable base space double cover and let $\tilde{\varphi}:G\rightarrow Diff_{+}^{fop}(\tilde{M})$ satisfy the obstruction condition. Then $\tilde{\varphi}:G\rightarrow Diff_{+}^{fop}(\tilde{M})$ is equivalent to an action that commutes with the covering translation $\tau:\tilde{M}\rightarrow\tilde{M}$ if and only if $\widehat{\tilde{\varphi}}_{1}(g)(u)=\epsilon(g)u$ for $\epsilon(g)=\pm1$.
\end{thm}

\begin{proof}
$\tau:\tilde{M}\rightarrow\tilde{M}$ is fiber-preserving, so we can consider the restricted map $\hat{\tau}:\widehat{\tilde{M}}\rightarrow\widehat{\tilde{M}}$. This leaves a product structure $k':S^{1}\times F\rightarrow\widehat{\tilde{M}}$ invariant by Theorem 4.4 in \cite{peet2018} (this result is an adaptation of Theorem 2.3 in P. Scott and W. Meeks' paper \textit{Finite group actions on 3-manifolds} \cite{Meeks1986}). So now consider the equivalent action defined by $k'\circ k^{-1}\circ\widehat{\tilde{\varphi}}(g)\circ k\circ k'^{-1}$. We then have that: $$k'^{-1}\circ(k'\circ k^{-1}\circ\widehat{\tilde{\varphi}}(g)\circ k\circ k'^{-1})\circ\hat{\tau}\circ k'=(k^{-1}\circ\widehat{\tilde{\varphi}}(g)\circ k)\circ(k'^{-1}\circ\hat{\tau}\circ k')$$ $$k'^{-1}\circ\hat{\tau}\circ(k'\circ k^{-1}\circ\widehat{\tilde{\varphi}}(g)\circ k\circ k'^{-1})\circ k'=(k'^{-1}\circ\hat{\tau}\circ k')\circ(k^{-1}\circ\widehat{\tilde{\varphi}}(g)\circ k)$$
Now: $$(k^{-1}\circ\widehat{\tilde{\varphi}}(g)\circ k)\circ(k'^{-1}\circ\hat{\tau}\circ k')(u,x)=(\widehat{\tilde{\varphi}}_{1}(g)(u^{-1}),\widehat{\tilde{\varphi}}_{2}(g)(\hat{\tau}_{B}(x)))$$ $$(k'^{-1}\circ\hat{\tau}\circ k')\circ(k^{-1}\circ\widehat{\tilde{\varphi}}(g)\circ k)(u,x)=(\widehat{\tilde{\varphi}}_{1}(g)(u)^{-1},\hat{\tau}_{B}(\widehat{\tilde{\varphi}}_{2}(g)(x)))$$
These are equal for all $(u,x)$ if and only if $\widehat{\tilde{\varphi}}_{1}(g)(u)=\epsilon(g)u$ for $\epsilon(g)=\pm1$ (noting that the orientation of the fibers is preserved) and $\tilde{\varphi}_{2}(g)$ commutes with the induced covering translation $\tau_{B}:\tilde{B}\rightarrow\tilde{B}$. This second is true again by \cite{tucker1983finite}. 

By extension across the fillings, it follows that there is an action equivalent to $\varphi$ that commutes with $\tau$.
\end{proof}

\begin{cor}
Let $M$ be an orientable Seifert fibered manifold that fibers over an orbifold that has underlying space the real projective plane $\mathbb{P}^{2}$ and obstruction class $b$. Let $q:\tilde{M}\rightarrow M$ be the orientable base space double cover and let $\tilde{\varphi}:G\rightarrow Diff_{+}^{fop}(\tilde{M})$ satisfy the obstruction condition. Then $\tilde{\varphi}:G\rightarrow Diff_{+}^{fop}(\tilde{M})$ is equivalent to an action that commutes with the covering translation $\tau:\tilde{M}\rightarrow\tilde{M}$ if and only if $\widehat{\tilde{\varphi}}_{1}(g)(u)=\epsilon(g)u$ for $\epsilon(g)=\pm1$.
\end{cor}

\begin{proof}
This follows directly as an application of Theorem 4.2.
\end{proof}

It can be seen that the obstruction condition will always be satisfied in the above corollary as the obstruction class for the orientable double cover is always even. This will be established in a future paper dealing with the specific case when the Seifert manifold fibers over an orbifold with underlying elliptic surface ($S^{2}$ or $\mathbb{P}^{2}$).

So we can now use these corollaries and Theorem 3.1 to see that given a finite fiber- and orientation-preserving action on an orientable Seifert manifold with non-orientable base space, the action can be constructed as a projected action of the extended product action type shown in \cite{peet2018}, provided the lifted action satisfies the obstruction condition.

We call such actions \textit{projected extended product actions}.

We give an example to illustrate this:

\begin{exmp}
We consider the prism manifold $M=(1,n_{2}|(q,p))$. This manifold fibers over the nonorientable orbifold $\mathbb{P}^{2}(m)$. $M$ is double covered by a lens space $\tilde{M}=(0,o_{1}|(q,p),(q,p))$. So then given an a fiber- and orientation-preserving action $\varphi: G \rightarrow Diff^{fp}_{+}(M)$, we can use Theorem 3.1 to derive a unique action $\tilde{\varphi}: G \rightarrow Diff^{fp}_{+}(\tilde{M})$ with $\tilde{\varphi}(G) \subset Cent^{fop}_{+}(\tau)$.

Then, noting that the obstruction class of $\tilde{M}$ is $2b$ if the obstruction class of $M$ is $b$, $\tilde{\varphi}$ can be constructed as an extended product action if it satisfies the obstruction condition.

Then $\varphi$ can be constructed as a projected extended product action.

\end{exmp}

\section{Group structures}

We now establish the specific structure of the groups that act in the constructed manner by using the following proposition from \cite{peet2018}.

\begin{prop}
 Suppose that $\varphi:G\rightarrow Diff(S^{1})\times Diff(F)$ is a finite group action with $\varphi(g)(u,x)=(\varphi_{S^{1}}(g)(u),\varphi_{F}(g)(x))$ such that $\varphi_{S^{1}}(g)$ is orientation-preserving if and only if $\varphi_{F}(g)$ is orientation-preserving. Suppose that there exists $g_{-}\in G$ such that $\varphi_{S^{1}}(g_{-})$ is orientation-reversing and $g_{-}^{2}=1$. Then $G$ is isomorphic to a subgroup of a semidirect product of $\mathbb{Z}_{n}\times\varphi_{F}(G)_{+}$ and $\mathbb{Z}_{2}$.
\end{prop}
From this, we yield the following corollaries:
\begin{cor}
 Suppose that $\varphi:G\rightarrow Diff(S^{1})\times Diff(F)$ is a finite group action with $\varphi(g)(u,x)=(\varphi_{S^{1}}(g)(u),\varphi_{F}(g)(x))$ such that both $\varphi_{S^{1}}$ and $\varphi_{F}$ are orientation-preserving. Then $G$ is isomorphic to a subgroup of $\mathbb{Z}_{n}\times\varphi_{F}(G)$.
\end{cor}
\begin{proof}
This follows directly from the proof of Proposition 5.1 in \cite{peet2018}.
\end{proof}
\begin{cor}
\label{cor:associativity2}
Suppose that $\varphi:G\rightarrow Diff(M)$ is a finite group action on an orientable Seifert manifold with a non-orientable base space. Then provided that the unique lifted group action $\tilde{\varphi}:G\rightarrow Diff(\tilde{M})$ satisfies the obstruction condition, $G$ is isomorphic to a subgroup of $\mathbb{Z}_{2}\times H$ where $H$ is a finite group that acts orientation-preservingly on the orientable base space of $\tilde{M}$.
\end{cor}
\begin{proof}
This follows from Corollary 5.2 and Theorem 4.2.
\end{proof}
\section{Summary}

To summarize this paper, we have presented a unique orientable base space double cover that allows us to extend our previous results. In particular we used this to show an isomorphism between $Cent_{+}^{fop}(\tau)$ to $Diff_{+}^{fp}(M)$. This in turn allowed us to consider the finite, fiber- and orientation-preserving actions on $M$ by considering the finite, fiber-orientation-preserving actions on $\tilde{M}$ that commute with the covering translation. Our final corollary then showed that the groups must be isomorphic to a subgroup of $\mathbb{Z}_{2}\times H$ where $H$ is a finite group that acts orientation-preservingly on the orientable base space of $\tilde{M}$.

\bibliographystyle{unsrt}
\bibliography{references}
 
\vspace{2cm}
    \hfill \textbf{Benjamin Peet} 
    
    \hfill Department of Mathematics
    
    \hfill St. Martin's University
    
    \hfill Lacey, WA 98503 
    
    \hfill \texttt{bpeet@stmartin.edu}

\end{document}